\documentclass{elsarticle}

\usepackage{amssymb}
\usepackage{amsmath}
\makeindex
\usepackage{amsthm}%Beweisumgebung

\newtheorem{lemma}{Lemma}[section]
\newtheorem{thm}[lemma]{Theorem}

\newtheorem{cor}[lemma]{Corollary}
\newtheorem{ex}[lemma]{Example}
\newtheorem{rem}[lemma]{Remark}
\newtheorem{hyp}[lemma]{Hypothesis}
\newtheorem{definition}[lemma]{Definition}

\begin{document}

\setlength{\parindent}{0mm}

\newcommand{\id}{\textrm{id}}
\newcommand{\K}{\textrm{K}}
\newcommand{\C}{\textrm{C}}
\newcommand{\D}{\textrm{D}}
\newcommand{\FO}{\textrm{Fix}_\Omega}

\newcommand{\Aut}{\textrm{Aut}}
\newcommand{\Part}{\textrm{Part}}
\newcommand{\Sym}{\textrm{Sym}}
\newcommand{\Supp}{\textrm{Supp}}
\newcommand{\s}{\mathcal{S}}
\newcommand{\N}{\mathbb{N}}
\newcommand{\Lab}{\mathcal{L}}
\newcommand{\la}{\langle}
\newcommand{\ra}{\rangle}

\begin{center}
\Large{\textbf{Computational aspects of orbital graphs}}

\vspace{0.2cm} \small{Paula H\"ahndel, Christopher Jefferson, Markus Pfeiffer and Rebecca Waldecker}
\end{center}

\vspace{2cm}
\begin{center} \textbf{Abstract}

\end{center}

We introduce orbital graphs and discuss some of their basic properties. Then we focus on their usefulness for search algorithms for permutation groups, including finding the intersection of groups and the stabilizer of sets in a group.

\section{Introduction}

Orbital graphs are a well-known class of directed graphs coming from permutation groups; they are mentioned for example in \cite{C} and \cite{DM}.

These graphs have been considered by Heiko Thei{\ss}en in \cite{Th} for the computation of normalizers of subgroups of permutation groups, but our work does not build on his -- partly because our hypothesis is more general and partly because his results have not been published except for in his PhD thesis.

But our motivation resembles Thei{\ss}en's and is mainly computational; for example
we use orbital graphs in \cite{JPW} for refinements of partitions in order to improve search algorithms that are based on partition backtrack methods. We are convinced that the application in \cite{JPW} is just the beginning and that huge computational benefits can be gained from a better understanding of these graphs in the future.
In this article we classify those orbital graphs that are useless for computational purposes and we describe ways to 
detect these so-called ``futile'' graphs before they are even explicitly constructed. This way orbital graphs can be used most effectively in algorithms.

For the relevant notation we refer to \cite{C} and \cite{DM}, in particular for orbits, point stabilizers etc., but we introduce everything that might not be standard.
In Section 2 we discuss some basic theoretical results about orbital graphs that are probably well-known but that, to our knowledge, are mostly not contained in the existing literature. Then in Section 3 we prove specific new results that are motivated from a computational perspective and we finish with some open questions.

\section{Basic Properties of orbital graphs}

We begin by defining orbital graphs, by explaining some examples and by proving some basic properties.
Our notation is standard, and we point out that by a \textbf{proper digraph} we mean a digraph that has at least one arc such that there its reverse arc is not in the graph.
All digraphs considered here have no multiple arcs and no loops.

\begin{definition}
Let $H$ be a group of permutations on a set $\Omega$ and let $\alpha,\beta \in \Omega$ be distinct elements.
Then the \textbf{orbital graph of $H$ with base-pair $(\alpha,\beta)$} is defined in the following way:

The vertex set is $\Omega$ and the arc set is
$\{(\alpha^h,\beta^h) \mid h \in H\}$, where for all $h \in H$ and $\omega \in \Omega$ we denote by $\omega^h$ the image of $\omega$ under the permutation $h \in H$.

Once $H$ and its action on $\Omega$ are given, we denote the orbital graph of $H$ with base-pair $(\alpha,\beta)$ by
$\Gamma(H, \Omega,(\alpha,\beta))$.\\

One more general definition:

An \textbf{isolated vertex} of a digraph is a vertex with no arcs going into it or coming out of it.

Following \cite{C} and \cite{DM} we say that an orbital graph is \textbf{self-paired} if and only if, for all $\gamma, \delta \in \Omega$, it is true that $(\gamma,\delta)$ is an arc if and only if $(\delta,\gamma)$ is an arc.

\end{definition}

\begin{ex}
For symmetric groups in their natural action we always obtain a complete self-paired digraph as orbital graph, independent of the base-pair.

Next we let $H:=\la (23),(46)\ra \le \s_7$ and $\Omega:=\{1,2,3,4,5,6,7\}$.

If we choose the base-pair $(1,7)$, then this is the only arc in the orbital graph and the points $2,3,4,5,6$ are isolated.

The base-pair $(1,3)$ gives an orbital graph with arcs $(1,3)$ and $(1,2)$, and we obtain a maximum number of arcs by choosing the base-pair $(3,4)$. Then we have four arcs in total, namely $(3,4),(2,4),(3,6)$ and $(2,6)$.\\
\end{ex}

Some properties of orbital graphs can be found in \cite{C} and \cite{DM}, but we decided to include short proofs for the statements in the next lemma in order to make this article more self-contained.

\begin{hyp}\label{orbhyp}
Let $\Omega$ be a finite set, let $H \le G:=\Sym(\Omega)$ and let $\alpha,\beta \in \Omega$ be distinct. Let
$\Gamma:=\Gamma(H, \Omega,(\alpha,\beta))$ and let $A$ denote the set of arcs of $\Gamma$.
\end{hyp}

\begin{lemma}\label{basic}
Suppose that Hypothesis \ref{orbhyp} holds.
Then we have the following:

\begin{enumerate}
\item[(i)]
$\Gamma=\Gamma(H,\Omega,(\gamma,\delta))$ if and only if $(\gamma,\delta) \in A$.

\item[(ii)]
$\Gamma$ is self-paired if and only if some $h \in H$ interchanges $\alpha$ and $\beta$.

\item[(iii)]
$\alpha^H$ is precisely the set of vertices of $\Gamma$ that are the starting point of some arc.

\item[(iv)]
$\beta^H$ is precisely the set of vertices of $\Gamma$ that are the end point of some arc.

\item[(v)]
The number of arcs starting at $\alpha$ is $|\beta^{H_\alpha}|$ and
the number of arcs going into $\beta$ is
$|\alpha^{H_\beta}|$.

\end{enumerate}
\end{lemma}

\begin{proof}
(i) If $\Gamma=\Gamma(H,\Omega,(\gamma,\delta))$, then by definition $(\gamma,\delta)$ is an arc in $\Gamma$.

Conversely, suppose that $(\gamma,\delta)$ is an arc in $\Gamma$. Then there exists some $h \in H$ such that
$(\alpha^h,\beta^h)=(\gamma, \delta)$. Hence the orbital graph generated by $(\alpha,\beta)$ is the same as the orbital graph generated by $(\gamma, \delta)$.

%Hence, if $\omega \in \Omega$, then the statements $\omega \in \alpha^H$ and $\omega \in \gamma^H$ are equivalent. Similarly $\delta^H=\beta^H$ and hence $(\gamma,\delta)$ can be chosen as a base-pair instead of $(\alpha,\beta)$.\\

(ii)
By (i) $\Gamma$ coincides with $\Gamma(H,\Omega, (\beta, \alpha))$ if and only if the arc $(\beta,\alpha)$ exists in $\Gamma$, which happens if and only if
there exists some $h \in H$ such that $\alpha^h=\beta$ and $\beta^h=\alpha$.\\

(iii)
Let $\gamma \in \alpha^H$ and let $h \in H$ be such that $\gamma=\alpha^h$.
Then $(\gamma,\beta^h)$ is an arc with starting point $\gamma$.

Conversely, if $\delta \in \Omega$ is such that $(\gamma,\delta)$ is an arc in $\Gamma$, then there exists some $h \in H$ such that $(\alpha^h,\beta^h)=(\gamma,\delta)$ and hence $\gamma=\alpha^h \in \alpha^H$.

Similar arguments show (iv).\\

(v) We just calculate that the number of arcs starting at $\alpha$ is

$|\{(\alpha,\gamma) \mid \gamma \in \Omega\}|=|\{(\alpha^h,\beta^h) \mid h \in H, \alpha^h=\alpha\}|=|\beta^{H_\alpha}|.$

The number of arcs going into $\beta$ is, by the same reasoning,

$|\{(\delta,\beta) \mid \delta \in \Omega\}|=|\{(\alpha^h,\beta^h) \mid h \in H, \beta^h=\beta\}|=|\alpha^{H_\beta}|.$\\
\end{proof}

\begin{rem}\label{rem} Some comments:

(a) Parts (iii) and (v) of the lemma, together, give the total number of arcs in $\Gamma$.
The number of arcs starting at $\alpha$ is exactly $|\beta^{H_\alpha}|$, so we obtain
 $|A|=|\alpha^H| \cdot |\beta^{H_\alpha}|$.

(b)
In (ii) it is not true that $H$ must contain the transposition $(\alpha, \beta)$.
A counterexample is provided by
$H:=\la (12)(34)\ra \le \s_4$ acting naturally on $\Omega:=\{1,2,3,4\}$ and its orbital graph
with base-pair $(1,2)$.
\end{rem}

\begin{lemma}\label{iso}
Suppose that Hypothesis \ref{orbhyp} holds.
Then the following are equivalent:

(a) ~$\Gamma$ has no isolated vertices.

(b) ~$\Omega \subseteq \alpha^H \cup \beta^H$.

\end{lemma}

\begin{proof}

If $\Gamma$ has some isolated vertex $\omega \in \Omega$, then there is no arc starting at $\omega$ and no arc ending there. By Lemma \ref{basic}~(ii) and (iii) it follows that $\gamma \notin \alpha^H \cup \beta^H$.
Hence (b) implies (a).

Conversely we suppose that (a) holds and we let $\gamma \in \Omega$.
As $\gamma$ is not isolated, it is the starting point of some arc or the end point. In the first case $\gamma \in \alpha^H$ and in the second case $\gamma \in \beta^H$ by Lemma \ref{basic}. It follows that $\Omega \subseteq \alpha^H \cup \beta^H$ as stated in (b).
\end{proof}

\begin{cor}\label{iso2}
Suppose that Hypothesis \ref{orbhyp} holds.

If $H$ acts transitively on $\Omega$, then $\Gamma$ has no isolated vertices.
If
$H$ is not transitive and $\Omega = \alpha^H \cup \beta^H$, then $\Gamma$ is bipartite and has no isolated vertices.
\end{cor}

\begin{proof}
The first statement follows from Lemma \ref{iso} because then $\Omega=\alpha^H=\beta^H$.

The second statement follows from the same lemma together with Lemma \ref{basic}~(iii) and (iv).
\end{proof}

\begin{lemma}\label{act}
Suppose that Hypothesis \ref{orbhyp} holds. Then $H$ acts on $\Gamma$ as a group of graph automorphisms.
\end{lemma}

\begin{proof}
Of course $H$ acts faithfully on the set $\Omega$ which is the vertex set of $\Gamma$.
For all vertices $\gamma \in \Gamma$ and all $h \in H$ we write $\gamma^h$ for the image of $\gamma$ under $h$ in the original permutation action.
Let $\gamma,\delta \in \Omega$.

If $(\gamma,\delta)$ is an arc, then there exists some $h \in H$ such that $(\gamma,\delta)=(\alpha^h,\beta^h)$ by definition of $\Gamma$.
Hence $(\gamma^g,\delta^g)=(\alpha^{hg},\beta^{hg})$ is an arc.

Conversely, if $(\gamma^h,\delta^h)$ is an arc, then there exists some $a \in H$ such that $(\gamma^h,\delta^h)=(\alpha^a,\beta^a)$ and hence $(\gamma,\delta)=(\alpha^{ah^{-1}},\beta^{ah^{-1}})$ is an arc.

As $H$ is a group, the induced maps are bijective and hence
every $h \in H$ induces a graph automorphism on $\Gamma$.
\end{proof}

\begin{lemma}\label{transimp}
Suppose that Hypothesis \ref{orbhyp} holds and let \(\Delta\) denote the connected component that contains \((\alpha, \beta)\).
Then every connected component of \(\Gamma\) that has size at least $2$ is isomorphic to \(\Delta\).
\end{lemma}

\begin{proof}
Let \({\Delta}_2\) denote an arbitrary connected component of \(\Gamma\) of size at least $2$ and let \((\gamma,\delta)\) be an arc in $\Delta_2$.

From the definition of orbital graphs let \(h\in H\) be such that \((\alpha^h,\beta^h)=(\gamma,\delta)\). Then \(h\) induces an automorphism on \(\Gamma\) by Lemma \ref{act} and it moves all arcs from \(\Delta\) to arcs in \(\Delta_2\).

Conversely, \(h^{-1}\) induces an automorphism on \(\Gamma\) that moves all arcs of \(\Delta_2\) into \(\Delta\). Thus it follows that \(\Delta\) and \(\Delta_2\) are isomorphic as graphs.
\end{proof}

Lemma \ref{lem:orbitalorbits} shows how to generate a set of base-pairs which generate all orbital graphs for a group \(H\). Parts (i) and (ii) show to take a representative from each orbit of \(H\) as the first element of the base-pair, and then part (iii) shows we must stabilize this point in \(H\), and take a representative from each orbit in this stabilizer for the second element of our base-pair. These base-pairs will allow us to analyse the set of orbital graphs of a group, before we construct any orbital graphs explicitly.

\begin{lemma}\label{lem:orbitalorbits}
Let $\Omega$ be a finite set and let $H \le G:=\Sym(\Omega)$.

\begin{enumerate}[(i)]
\item  Suppose that $\alpha,\beta \in \Omega$ and $\alpha \in \beta^H$. Then the set of orbital graphs
  of $H$ with base-pairs starting with $\alpha$ is equal to the set of orbital
  graphs of $H$ with base-pairs starting with $\beta$.

\item Suppose that $\alpha,\gamma \in \Omega$ and $\alpha \notin \gamma^H$. Then the set of orbital
  graphs of $H$ with base-pairs starting with $\alpha$ is disjoint from the set of
  orbital graphs of $H$ with base-pairs starting with $\gamma$.

\item Suppose that $\alpha,\beta,\gamma \in \Omega$ and that $\alpha \neq \beta$, $\alpha \neq \gamma$. Let $\Gamma_1:=\Gamma(H,\Omega,(\alpha,\beta))$ and $\Gamma_2:=\Gamma(H,\Omega,(\alpha,\gamma))$.
Then $\Gamma_1=\Gamma_2$ if and only if
  $\gamma \in \beta^{H_\alpha}$.
\end{enumerate}
\end{lemma}

\begin{proof}
\begin{enumerate}[(i)]
\item Let $h \in H$ be such that $\alpha^h=\beta$. Then for all $\gamma \in \Omega$,
  it follows that $\Gamma(H,\Omega,(\alpha,\gamma))=\Gamma(H,\Omega,(\alpha^h,\gamma^h)) = \Gamma(H,\Omega,(\beta,\gamma^h))$. Conversely $\Gamma(H,\Omega,(\beta,\gamma))=\Gamma(H,\Omega,(\beta^{(h^{-1})},
  \gamma^{(h^{-1})})) = \Gamma(H,\Omega,(\alpha,\gamma^{(h^{-1})})$.

\item Suppose that the pairs $(\alpha,\beta)$ and $(\gamma,\delta)$ generate the same orbital graph
  of $H$. Then by Lemma \ref{basic}~(i) there is $h \in H$ such that $(\alpha^h,\beta^h) = (\gamma,\delta)$, which
  implies that $\alpha \in \gamma^H$. This proves the statement.

\item If $\Gamma_1=\Gamma_2$, then $(\alpha,\beta)$ and $(\alpha,\gamma)$ generate the same orbital graph. So by Lemma \ref{basic}~(i) there is $h \in H$ such that $(\alpha^h,\beta^h)=(\alpha,\gamma)$. This means that $\alpha^h=\alpha$ and therefore
 $h \in H_\alpha$, which implies that $\gamma \in \beta^{H_\alpha}$. Conversely, if $\gamma \in \beta^{H_\alpha}$ then there
 exists $h \in H_\alpha$ such that $(\alpha^h,\beta^h)=(\alpha,\gamma)$ and hence $\Gamma_1=\Gamma_2$.
\end{enumerate}
\end{proof}

After these preparatory results we can embark on the topic of usefulness of these graphs in algorithms.

\section{Usefulness of orbital graphs in algorithms}

Reasoning about arbitrary permutation groups is computationally extremely
expensive. Therefore, Leon’s partition backtrack algorithm (see \cite{L})
replaces groups with the stabilizer of an ordered orbit partition during search.
This can be seen as an approximation: A group G is a subgroup of the stabilizer
of its ordered orbit partition in any supergroup of G. Using the intersection of
the automorphism groups of all orbital graphs instead gives a smaller group and
hence a faster algorithm -- but there are exceptions where
this approach does not give any advantage.
This motivates the following definition:

\begin{definition}
Suppose that Hypothesis \ref{orbhyp} holds and that $P$ is an ordered orbit
partition of $H$. We denote the stabilizer of $P$ in $G$ by $G_P$ and we
emphasize that $G_P$ stabilizes every $H$-orbit (i.e. every cell of the ordered
partition $P$) as a set and that it acts as the full symmetric group on every
orbit.

We say that the orbital graph $\Gamma$ is \textbf{futile} if and only if $G_P$,
in its natural action on $\Omega$, induces graph automorphisms on
$\Gamma$.

Just as a reminder:

A digraph \(\Gamma=(V,A)\) is said to be a \textbf{complete digraph} if and only
if its set of arcs is $\{(\omega_1,\omega_2) \mid \omega_1,\omega_2 \in \Omega,
\omega_1 \neq \omega_2\}$.

$\Gamma$ is called a \textbf{complete bipartite digraph} if and only if there
exist pair-wise disjoint subsets $S,E$ of vertices such that $V$ is the disjoint
union of $S$ (the ``starting'' vertices) and $E$ (the ``end'' vertices) and the
set of arcs is exactly $A = \{(\omega_1,\omega_2) \mid \omega_1 \in S, \omega_2\in
E\}$.
\end{definition}

Our main theoretical result on this topic classifies futile orbital graphs. In
particular, Corollary \ref{cor:futilecheck} shows how an orbital graph can be
recognized as futile before it is even constructed. We note that the following
result does not place any restrictions about the number of orbits of $H$ on
$\Omega$. In particular there could be arbitrarily many isolated points in
$\Gamma$.

\begin{thm}\label{use}
Suppose that Hypothesis \ref{orbhyp} holds.
Then $\Gamma$ is futile if and only if it has a unique connected component
$\Delta$ of size at least $2$ and moreover one of the following holds:

(a) $\Delta$ is a complete digraph or

(b) $\Delta$ is a complete bipartite digraph.
\end{thm}

\begin{proof}
Let $P$ be an ordered orbit partition of $H$. $G_P$ acts on the set of orbits of
$H$ and it acts faithfully on the set of vertices of $\Gamma$. Hence to answer
if $\Gamma$ is futile or not, we only have to consider arcs in $\Gamma$.

We split our proof into two cases depending on whether or not $\Gamma$ is a
proper digraph.

\textbf{Case 1:} $\Gamma$ is a proper digraph.

Then $\Gamma$ is not self-paired and Lemma \ref{basic}~(i) and (ii) imply that,
for all $\omega_1, \omega_2 \in \Omega$, there is at most one arc between them.
In the following arguments we will often refer to Lemma \ref{basic}~(iii) and
(iv) as well.

We begin with the hypothesis that $\Gamma$ is futile.

\textbf{(1)} Suppose that $\gamma,\delta \in \Omega$ are distinct and in the
same $H$-orbit. Then they are not on an arc. In particular $\alpha^H \neq
\beta^H$.

\begin{proof}
As $\gamma$ and $\delta$ are in the same $H$-orbit, they lie in the same cell of
the partition $P$. It follows from the futility of $\Gamma$ that
the transposition $(\gamma,\delta) \in \Sym(\Omega)$, which stabilizes $P$,
induces a graph automorphism on $\Gamma$. Therefore neither $(\gamma,\delta)$
nor $(\delta,\gamma)$ is an arc. From this and the fact that $(\alpha, \beta)
\in A$ it follows that $\alpha^H \neq \beta^H$.
\end{proof}

\textbf{(2)} Suppose that $\omega \in \Omega$ is on an arc. Then it is either a
starting point or an end point, but not both.

\begin{proof}
This follows from Lemma \ref{basic}~(iii) and (1).
\end{proof}

Let $S:=\alpha^H$ and $E:=\beta^H$, and let $I \subseteq \Omega$ denote the set of isolated vertices of $\Gamma$.

\textbf{(3)} $\Omega=S \dot \cup E \dot \cup I$. Moreover $S \cup E$ spans the
unique connected component of $\Gamma$ of size at least $2$, and this component
is a complete bipartite digraph.

\begin{proof}
The first statement follows from (2). Moreover there are no arcs between
vertices in $S$ or $E$, respectively, by (1). We show that all elements of $E$
are on an arc with $\alpha$:

For all $\gamma \in E$, we find the transposition $g:=(\beta,\gamma) \in G_P$,
and it fixes $\alpha^H$ point-wise by (1). The futility of $\Gamma$ implies that
$g$ maps the arc $(\alpha,\beta)$ to the arc $(\alpha,\gamma)$. Now it follows
that $A=S \times E$ and hence the digraph spanned by $S \cup E$ is a complete
bipartite digraph, and it is the unique connected component of size at least $2$
of $\Gamma$.
\end{proof}

Conversely, we suppose that $\Gamma$ has a unique connected component of size at
least $2$ and that this component is a complete bipartite digraph. We prove that
$\Gamma$ is futile.

Let $S$ and $E$ denote the subsets of the vertex set of $\Gamma$ such that all
arcs start at $S$ and end at $E$. Let $I$ be the set of isolated vertices of
$\Gamma$, so that $\Omega=S \dot \cup E \dot \cup I$.

Now $\alpha^H \subseteq S$ and the bipartite structure implies that even $\alpha^H=S$.
Similarly $\beta^H =E$. Therefore $G_P$ stabilizes the sets $S$, $E$ and $I$. We
already know that $G_P$ permutes the vertices of $\Gamma$ faithfully, so now we
look at arcs.

Let $g \in G_P$ and let $(\omega_1,\omega_2) \in A$. Then $\omega_1 \in S$,
$\omega_2 \in E$ and there exists some $h \in H$ such that
$(\alpha^h,\beta^h)=(\omega_1,\omega_2)$. Since $G_P$ stabilizes the sets $S$
and $E$, we see that $\omega_1^g \in S$ and $\omega_2^g \in E$. The completeness
property then implies that $(\omega_1^g,\omega_2^g) \in A$.

Conversely, if $(\omega_1^g,\omega_2^g) \in A$, then there exists some $h \in H$
such that $(\alpha^h,\beta^h)=(\omega_1^g,\omega_2^g)$. Now
$\omega_1=\alpha^{hg^{-1}} \in S$ and $\omega_2=\beta^{hg^{-1}} \in E$ whence
$(\omega_1,\omega_2) \in A$ by completeness.

Hence $\Gamma$ is futile.\\

\textbf{Case 2:} $\Gamma$ is not a proper digraph, which means that it is self-paired.

We begin, once more, with the hypothesis that $\Gamma$ is futile. Let $\Delta$
denote the connected component of $\Gamma$ that contains the base-pair
$(\alpha,\beta)$ and let $\gamma \in \Omega$ be an arbitrary, non-isolated
vertex.

We know that $\beta^H=\alpha^H$ by Lemma \ref{basic}~(iii) and (iv), because
$\Gamma$ is self-paired. Since some arc starts or ends in $\gamma$, we also have
that $\gamma \in \alpha^H$ and hence $\alpha,\beta,\gamma$ are all in the same
$H$-orbit and hence in a common cell of the partition $P$.
In particular the transposition $g:=(\beta,\gamma)$ is contained in $G_P$ and
because of the futility it induces an automorphism on $\Gamma$.

Then $(\alpha,\beta) \in A$ implies that $(\alpha,\gamma)=(\alpha^g,\beta^g) \in
A$. This argument shows that $\Delta$ is a complete digraph and that it is the
only connected component of size at least $2$ in $\Gamma$.\\

We conversely suppose that $\Gamma$ has a unique connected component of size at
least $2$ and that it is complete. Together with the definition of orbital
graphs (and the fact that arcs always go both ways in the present case) this
implies that $\alpha^H=\beta^H$ spans the unique connected component of size at
least $2$ and that the isolated vertices, viewed as elements of $\Omega$, are
not contained in $\alpha^H$.

We know that $G_P$ acts faithfully on the vertex set of $\Gamma$. Now let $g \in
G_P$ and let $\omega_1,\omega_2 \in \Omega$. We recall that $\alpha^H=\beta^H$
is $G_P$-invariant.

Then it follows as in Case 1, using the completeness, that $(\omega_1,\omega_2)
\in A$ if and only if $(\omega_1^g,\omega_2^g) \in A$. Consequently $G_P$ acts as a group of
automorphisms on $\Gamma$, i.e. $\Gamma$ is futile.
\end{proof}

We give an example in order to illustrate that futility of an orbital graph is not obvious and why further investigations into the usefulness of orbital graphs should be pursued.

\begin{ex}\label{2K3}
We let $G:=\s_9$ and we look at the subgroup

$H:=\langle
(12),(13),(45),(46),(14)(25)(36),(789)\rangle$. Let $\Gamma$ be the
orbital graph for $H$ with base-pair $(1,2)$. Then $\Gamma$ has the following shape:

On the vertices $1,2,3$ and $4,5,6$ we have a complete digraph, respectively,
there is no arc between the sets $\{1,2,3\}$ and $\{4,5,6\}$, and the points
$7,8$ and $9$ are isolated. This might look like a futile graph, but according
to the theorem it is not. Consider an ordered orbit partition
$P:=[1,2,3,4,5,6|7,8,9]$ of $H$.

The group $G_P$ contains the transposition $(24) \in G$.
This element interchanges the vertices $2$ and $4$ of $\Gamma$ and fixes $1$, so
this element does not induce an automorphism on $\Gamma$. (Otherwise the arc
$(1,2)$ would be mapped to the arc $(1,4)$, which does not exist). This graph
can be used to deduce, for example, that any element which swaps $1$ and $4$ must also
swap $\{2,3\}$ with $\{5,6\}$.

Hence $G_P$ does not act as a group of automorphisms on $\Gamma$ and we see that
$\Gamma$ is, in fact, not futile.
\end{ex}

It is important that we can detect futile graphs easily, without having to create them
explicitly. We will now give a collection of Lemmas which allow futile orbital graphs to
be detected using only information about orbits and stabilizers of a group, without
explicit construction of entire orbital graphs.

\begin{lemma}\label{allplnontrans}
Suppose that Hypothesis \ref{orbhyp} holds and that
\(\Omega=\alpha^H\dot{\cup}\beta^H\dot{\cup} I\), where $I \subseteq \Omega$ is the set of isolated vertices of $\Gamma$.
Then \(\Gamma\) is futile if and only if \(H_{\alpha}\) acts transitively on \(\beta^H\).

\end{lemma}

\begin{proof}
Suppose that $\Gamma$ is futile. Then Theorem \ref{use} and Lemma \ref{basic}~(iii) and (iv) imply that
\(\Gamma\) is a complete bipartite digraph.

In particular, for all \(\delta \in \beta^H\) it follows that \((\alpha,
\delta)\in A\) and so there exists some \(h\in H\) such that
$(\alpha,\delta)=(\alpha^h,\beta^h)$. In particular $H_\alpha$ is transitive on
$\beta^H$.
Conversely we suppose that $H_\alpha$ is transitive on $\beta^H$. Hence for all
\(\beta' \in \beta^H\) there exist some \(h\in H_\alpha\) such that
$\beta^h=\beta'$.

We prove that $H_\beta$ acts transitively on $\alpha^H$, so we let $\alpha' \in
\alpha^H$ and we choose $g \in H$ such that $\alpha'=\alpha^g$. Then, using the
transitivity argument above, we let $h \in H_\alpha$ be such that
$\beta^h=\beta^{g^{-1}}$. Then $\beta^{hg}=\beta$ and $\alpha^{hg}=\alpha'$,
which shows the transitive action of $H_\beta$ on $\alpha^H$.

Now the definition of an orbital graph implies that $\Gamma$ is a complete
bipartite digraph and hence futile, by Theorem \ref{use}.
\end{proof}

We finish by giving some concrete bounds on the number of edges in futile and
non-futile orbital graphs.

\begin{lemma}\label{altcounting}
  Suppose that Hypothesis \ref{orbhyp} holds. Let $n = |\alpha^H|$, $m =
  |\beta^H|$, and $I \subseteq \Omega$ be the set of isolated vertices of $\Gamma$.

  Then $\Gamma$ is futile if one of the following hold.
  \begin{enumerate}[i)]
  \item $\beta \in \alpha^H$ and $\Gamma$ has strictly more than \(n(n-2)\)
    arcs.
  \item $\Omega = \alpha^H \dot\cup \beta^H \dot\cup I$ and $\Gamma$ has
    strictly more than \(n(m-1)\) or \(m(n-1)\) arcs.
  \end{enumerate}
\end{lemma}

\begin{proof}
  To prove (i) suppose that \(\gamma,\delta \in \alpha^H\) are distinct and such
  that \((\gamma, \delta) \notin A\). Let \(r\) be the number of arcs starting in
  \(\gamma\). Now \((\gamma,\gamma)\) and \((\gamma,\delta)\) are not in $A$, so
  it follows that \(r\leq n-2\).
  We recall that $H$ is transitive on $\alpha^H$, and hence all connected
  components of $\Gamma$ have size at least $2$, by Corollary \ref{iso2}. % ??

  In particular $\gamma$ is contained in a connected component of $\Gamma$ of
  size at least two, so we deduce from Lemma \ref{basic}~(iii) and (iv) and Lemma
  \ref{transimp} that for every vertex of \(\Gamma\), the number of arcs starting
  there is \(r\). Consequently $|A|= n\cdot r\leq n\cdot (n-2)$. This means,
  conversely, that $\Gamma$ is a complete digraph on $\alpha^H$ as soon as it has
  strictly more than \(n\cdot (n-2)\) arcs.

  To show (ii) suppose that there are \(\gamma \in \alpha^H\) and \(\delta \in
  \beta^H\) such that \((\gamma, \delta)\notin A\). Let \(r\) be the number of
  arcs starting in \(\gamma\). As all arcs starting in \(\gamma\) end in a vertex of
  \(\beta^G\backslash \{\delta\}\) it follows that \(r\leq m-1\).
  Let $\omega \in \alpha^H$. Then it follows from Lemma \ref{basic}~(iii) that the
  number of arcs starting in \(\omega\) is
  $|\{(\omega_1,\beta^{g})\mid g\in H, \alpha^g=\omega_1\}|$.

  Hence $|A|=n\cdot r\leq n\cdot (m-1)$.

  By counting the number of arcs ending in some vertex we obtain, in a similar way, that
  $|A| \le m\cdot (n-1)$ as well.
  Hence if \(\Gamma\) has strictly more than \(n\cdot (m-1)\) or \(m\cdot (n-1)\)
  arcs, then \(\Gamma\) is a complete bipartite digraph.
\end{proof}

In practice we use Corollary \ref{cor:futilecheck}, which combines Lemma \ref{altcounting} with Remark
\ref{rem} to efficiently identify futile
orbital graphs before they are constructed.

\begin{cor}\label{cor:futilecheck}
Suppose that Hypothesis \ref{orbhyp} holds. Then $\Gamma$ is
futile if and only if one of the following conditions is true:

\begin{enumerate}[i)]
\item $\beta \in \alpha^H$, and $|\beta^{H_\alpha}| = |\alpha^H|$.
\item $\beta \not\in \alpha^H$ and $|\beta^{H_\alpha}| = |\beta^H|$.
\end{enumerate}
\end{cor}

\begin{proof}
\begin{enumerate}
\item[(i)] We are in case (i) of Lemma \ref{altcounting}. By Remark \ref{rem} the orbital graph
has size  $|\alpha^H| \cdot |\beta^{H_\alpha}|$. The only way this can be larger than $|\alpha^H|(|\alpha^H|-2)$
is if $|\beta^{H_\alpha}| + 1 \geq |\alpha^H|$. As $\beta \in \alpha^H$, $\beta^{H_\alpha}$ is a proper subset of $\alpha^H$ (the subset is proper because it does not contain $\alpha$). Therefore $|\beta^{H_\alpha}| + 1 = |\alpha^H|$.
\item[(ii)] We are in case (ii) of Lemma \ref{altcounting}. Again by Remark \ref{rem} the orbital graph
has size $|\alpha^H| \cdot |\beta^{H_\alpha}|$. The only way this can be larger than $|\alpha^H|(|\beta^H|-1)$
is if $|\beta^H| \leq |\beta^{H_\alpha}|$. As $\beta^H$ contains $\beta^{H_\alpha}$, this implies $|\beta^H| = |\beta^{H_\alpha}|$.
\end{enumerate}
\end{proof}

\subsection{Transitive Groups}

For transitive groups, it is simpler to identify the futile orbital graphs.
These were the first groups where we discovered that some orbital graphs are
more useful.

\begin{lemma}\label{everypointless}
Suppose that Hypothesis \ref{orbhyp} holds and that $H$ acts transitively on $\Omega$.
\begin{enumerate}[i)]
\item\label{ep1} If $H$ acts 2-transitively on $\Omega$, then $\Gamma$ is futile.
\item\label{ep2} If $\Gamma$ is futile, then $H$ acts 2-transitively on $\Omega$ (and hence all orbital graphs are futile).
\end{enumerate}
\end{lemma}

\begin{proof}
For (\ref{ep1}) we suppose that \(H\) acts 2-transitively on \(\Omega\). Then
whenever \(\gamma, \delta \in \Omega\) are distinct, there exits some \(h \in H\) such that \((\alpha^h,\beta^h)=(\gamma,\delta)\) and hence
\(\Gamma\) is a complete digraph.
By Theorem \ref{use} it follows that $\Gamma$ is futile.

For (\ref{ep2}) we suppose that \(\Gamma\) is futile and we deduce, again by
Theorem \ref{use}, that \(\Gamma\) is a complete digraph or a complete bipartite
digraph. The second case is impossible because \(H\) is transitive on
\(\Omega\). So \(\Gamma\) is a complete digraph and for any two distinct
elements \(\gamma, \delta \in \Omega\), we deduce that \((\gamma, \delta) \in
A\). Then by definition of an orbital graph, there is \(h\in H\) such that
\((\alpha^h,\beta^h)=(\gamma,\delta)\). Hence \(H\) acts 2-transitively on
\(\Omega\) and the last statement follows from (\ref{ep1}).
\end{proof}

So we see that for transitive groups if one orbital graph is futile, then all of
them are. Lemma \ref{everypointless} lets us quickly detect this, as the level of transitivity of
a group can be efficiently calculated.

In general these bounds from the lemmas cannot be improved, as the following example shows.

\begin{ex}
Let \(\Omega:=\{1,2,3,4\}\) and \(G:=\langle(1243),(12)(34)\rangle \leq
\Sym(\Omega) \). Then the orbital graph \(\Gamma:=\Gamma(G,(1,2),\Omega)\) has
exactly the \(8=4\cdot(4-2)\) arcs \((1,2),(2,1),(1,3),(3,1),(2,4),(4,2),(3,4)\)
and \((4,3)\) which is the bound in Lemma \ref{altcounting} (i).

Let \(\Omega:=\{1,2,3,4,5,6\}\) and \(G:=\langle(123)(456),(13)(45)\rangle \leq
\Sym(\Omega) \). Then the orbital graph \(\Gamma:=\Gamma(G,(1,4),\Omega)\) has
exactly the \(6=3\cdot(3-1)\) arcs \((1,4),(1,6),(2,4),(2,5),(3,5),\) and
\((3,6)\) which is the bound in Lemma \ref{altcounting} (ii).
\end{ex}

%From a computational perspective, this is not problematic: Highly transitive
%groups are well-understood and often our methods will still be useful when the
%algorithm considers subgroups that are not as highly transitive.

\section{Concluding remarks}

As mentioned earlier, this work on orbital graphs is motivated by applications in search algorithms for permutation groups. A systematic approach is needed for many open problems and potential applications, and orbital graphs are an interesting class of graphs in their own right. Therefore we phrase some questions, with only some of them being directly related to applications.

\begin{itemize}
\item
Instead of just separating the futile graphs from the useful ones for our algorithms, is it possible to create a finer distinction?

\item
Higman's Theorem (see for example p.68 in \cite{DM}) says that for transitive groups, primitivity can be detected from the orbital graphs.

For imprimitive groups, being able to detect blocks quickly and bring them into a ``usefulness analysis'' of the graph would be beneficial for computational questions (see Exercise 3.2.14 in \cite{DM}).

\item
The theory of association schemes seems to be closely related to orbital graphs. What applications does this have in computational algebra and how do our results relate to this?
\end{itemize}

There is work in progress on most of these questions.\\

\textbf{Acknowledgements.}

All authors thank the DFG (\textbf{Wa 3089/6-1}) and the EPSRC CCP CoDiMa
(\textbf{EP/M022641/1}) for supporting this work. The second author would
like to thank the Royal Society, and the EPSRC (\textbf{EP/M003728/1}).
The third author would
like to acknowledge support from the OpenDreamKit Horizon 2020 European Research
Infrastructures Project (\#676541).

\end{document}